\documentclass[11pt]{amsart}

\theoremstyle{plain}

\newtheorem{theorem}{Theorem}
\newtheorem{lemma}{Lemma}
\newtheorem{corollary}{Corollary}

\theoremstyle{remark}
\newtheorem*{remark}{Remark}

\newtheorem*{question}{Question}
\newtheorem{definition}{Definiton}

\usepackage{amssymb}
\usepackage{amsthm}

\title[Milnor-Wood inequality]{A Milnor--Wood inequality for complex hyperbolic lattices in quaternionic space}
\author[Garc\'ia Prada]{Oscar Garc\'ia-Prada}
\address{ Instituto de Ciencias Matem\'aticas CSIC-UAM-UC3M-UCM \\
Serrano 121\\ 28006 Madrid, Spain}
\email{oscar.garcia-prada@icmat.es} 

\author[Toledo] {Domingo Toledo}
\address{ Department of Mathematics \\
University of Utah\\
Salt Lake City, UT 84112 }
\email{toledo@math.utah.edu}
\date{\today} 

\thanks{First author partially supported by Ministerio de Ciencia e Innovaci\'on (Spain) Grant MTM2007-67623. Second author supported by National Science 
Foundation Grant DMS-0600816.}

\begin{document}
\maketitle

\section{Introduction}

The purpose of this paper is to prove a Milnor-Wood inequality for representations of the fundamental group of a compact complex hyperbolic manifold in the group of isometries of quaternionic hyperbolic space.  Of special interest is the case of equality, and its application to rigidity.  We show that equality can only be achieved for totally geodesic representations, thereby establishing a global rigidity theorem for totally geodesic representations.   We point out the recent papers of Kim-Pansu-Klingler \cite{KPK} and Klingler \cite{Kl} where closely related local rigidity theorems for totally geodesic representations  are proved.

To explain our results, let $m\ge 2$ and let  $\Gamma\subset SU(m,1)$ be a co-compact,  torsion-free lattice.  Then $X = \Gamma\backslash B^m$ is a compact complex manifold covered by the unit ball $B^m\subset\mathbb{C}^m$.   (We will write either $B^m$ or $H_\mathbb{C}^m$ for the complex hyperbolic $m$-space, which is the symmetric space of $SU(m,1)$.)  The space of invariant two forms on $B^m$ is one-dimensional.  Choose a generator $\omega = \omega_{B^m}$ to represent the Chern class of the tautological line sub-bundle.   This is a K\"ahler form for $B^m$ and descends to a K\"ahler form $\omega_X$ on $X$ .  Let 

$$
v(X) = \int_X\omega_X^m, 
$$  
Then $v(X)$ is a positive integer, proportional to the volume of $X$.   In terms of characteristic numbers and other invariants of $X$, $v(X) = (-1)^m\mathcal{T}(X)$ where $\mathcal{T}(X)$ is the Todd genus of $X$, in other words,  by the  Hirzebruch-Riemann-Roch theorem, $v(X) = (-1)^m  \chi(X,\mathcal{O}_X)$.

 Let $H^n_\mathbb{H}$ be the quaternionic hyperbolic space of quaternionic dimension $n$, with group of isometries $Sp(n,1)$.  The space of invariant four-forms on  $H^n_\mathbb{H}$ is one-dimensional.  We choose a non-zero form $\alpha$ in this space, and normalize it by requiring that  its restriction to a totally geodesic complex hyperbolic subspace $H^m_\mathbb{C} = B^m$ is  the form $\omega_{B^m}^2$, where $\omega_{B^m}$ is as above.

  Let $\rho:\Gamma\to Sp(n,1)$ be a representation.  Using the above definitions, we can assign to it a characteristic number $c(\rho)$ as follows.  Let $f:B^m \to H^n_\mathbb{H}$ be a smooth equivariant map (which exists and is unique up to equivariant homotopy).   Then $f^*\alpha$ is an invariant form on $B^m$, so it descends to a form on $X$.  
  We then define  $c(\rho)$ by
$$
c(\rho) = \int_X\ \omega_X^{m-2}\wedge  f^*\alpha \ .
$$
Our main result is the following theorem:

\begin{theorem}
\label{thm-main}
\begin{enumerate}
\item $|c(\rho)| \le v (X)$.
\item If equality holds, $\rho$ is a totally geodesic representation.
\end{enumerate}

\end{theorem}

By a {\em totally geodesic representation} we mean that there is a totally geodesic $H^m_\mathbb{C}\subset H^n_\mathbb{H}$ so that the image of the representation lies in the subgroup $G\subset Sp(n,1)$ that preserves this $H^m_\mathbb{C}$ and that the equivariant map $f$, which can be assumed to have image in this $H^m_\mathbb{C}$, is totally geodesic isometric embedding.

In group-theoretic terms, note that the subgroup $G$ of $Sp(n,1)$ that leaves $B^m\subset H^n_\mathbb{H}$ invariant is of the form $G_1\times G_2$ where $G_1$ is isomorphic to $SU(m,1)$ and $G_2$ is isomorphic to the compact group $U(1)\times Sp(n-m)$, the centralizer of $G_1$ in $Sp(n,1)$.    To say that $\rho$ is a totally geodesic representation is the same as saying that the image of $\rho$ lies in such a subgroup $G_1\times G_2$, and, splitting $\rho = \rho_1\times \rho_2$,  that $\rho_1$ is conjugate to the inclusion $\Gamma\subset SU(m,1)$.  In particular, a totally geodesic representation is faithful.

Since the characteristic number $c(\rho)$  is a topological invariant, therefore invariant under deformation, we obtain the following local rigidity theorem:

\begin{corollary}
Let $\rho:\Gamma\to Sp(n,1)$ be a totally geodesic representation.  Then any deformation $\rho_t$ of $\rho$ is a totally geodesic representation.  
\end{corollary}

For convenience, we have stated our results in terms of the groups $SU(m,1)$ and $Sp(n,1)$ that do not act effectively on the corresponding symmetric spaces.   We could easily modify our statements in terms of the quotient groups that act effectively.

\begin{question}  It is natural to ask for examples of representations $\rho$ where $0<c(\rho) < v(X)$.  In the case $m=2$ we can construct such examples by taking two ball quotients $X$, $Y$ with a surjective holomorphic map $f:X\to Y$ with $0< deg(f) < \frac{v(X)} {v(Y)} $ (these exist, see \cite{To2}), and a totally geodesic representation $\rho': \pi_1(Y)\to Sp(n,1)$.  Then the composition $\rho'\circ f_*:\pi_(X)\to Sp(n,1)$ satisfies the desired inequality.    But these representations do not give satisfactory examples, because, even though the representation is not totally geodesic,  its image still leaves invariant a totally geodesic  $H_\mathbb{C}^2\subset H_\mathbb{H}^n$.   Thus the real question is:  are there representations that satisfy $0 < c(\rho) < v(X)$ {\em and} whose image does not leave any geodesic $H_\mathbb{C}^m \subset H_\mathbb{H}^n$ invariant. 
\end{question}

We recall that the classical Milnor-Wood inequality concerns representations of surface groups in $SL(2,\mathbb{R})$ \cite{Mi}  or $Top(S^1)$  \cite{Wo}.   It has been generalized  to representations of surface groups in groups of isometries of Hermitian symmetric spaces, and the case of equality has received much attention 
\cite{BGG:2003,BGG:2005,BIW,GGM:2008}.  It is harder to prove such inequalities for representations of fundamental groups of higher dimensional manifolds.   For representations of fundamental groups of our $\Gamma$'s into $SU(n,1)$ there is the rigidity result of Corlette \cite{Co}, on which our result is modeled.    For some recent work on higher dimensional domains, see \cite{BCG, BG, KM1, KM2},  and references in these papers.

\begin{question}
A very natural question is whether there is a Milnor-Wood inequality for representations of lattices in $Sp(1,1)$, (equivalently,  represenations of fundamental groups of constant curvature four-manifolds), in $Sp(n,1), \ n>1$.   It is conjectured that, in analogy with  \cite{To1},  a sharp inequality exists, and that the case of equality characterizes the totally geodesic representations.  A suggestive local rigidity result compatible with this conjecture has been proved in \cite{KP}.
\end{question}

\section{Quaternionic Hyperbolic Space and its Twistor Space}

Let us write $\mathbb{H}$ for the quaternions and $\mathbb{H}^{n+1}$ for a right - quaternionic vector space of quaternionic dimension $n+1$.   We will use the quaternionic hermitian forms 
\begin{equation*}
h_0^\mathbb{H}(X,Y) =  \bar x_1 y_1 + \dots + \bar x_n y_n + \bar x_{n+1} y_{n+1} 
\end{equation*}
 and  
\begin{equation*}
h^\mathbb{H} (X,Y) =  \bar x_1 y_1 + \dots +\bar x_n  y_n -  \bar x_{n+1} y_{n+1} 
\end{equation*} 
on $\mathbb{H}^{n+1}$, where $X = (x_1,\dots,x_{n+1}),\  Y = (y_1\dots y_{n+1})\in \mathbb{H}^{n+1}$.    The quaternionic projective space $\mathbb{H}P^n$ is the manifold of right quaternionic lines in $\mathbb{H}^{n+1}$.   The group $Sp(n+1)$ of right quaternionic linear isometries of $h_0^\mathbb{H}$ acts transitively on it (by left multiplication by suitable quaternionic matrices) and the line $(0,\dots,0,1)$ has isotropy group $Sp(n)\times Sp(1)$, thus $\mathbb{H}P^n = Sp(n+1)/Sp(n)\times Sp(1)$.

The quaternionic hyperbolic space $H^n_\mathbb{H}$ is the open subset of $\mathbb{H}P^n$ consisting of those right-quaternionic lines $L$ on which the form $h^\mathbb{H}$ is negative: $h^\mathbb{H}(X,X)<0$ for all $X\in L$.  It is a homogeneous space for the group $Sp(n,1)$ of right quaternionic linear maps that preserve the form $h^\mathbb{H}$, and the isotropy group of the line $(0,\dots,0,1)$ is again $Sp(n)\times Sp(1)$.  We will write  $H_\mathbb{H}^n$ for the quaternionic hyperbolic space.   Thus $H_\mathbb{H}^n = Sp(n,1)/Sp(n)\times Sp(1)$.

The algebra of $Sp(n,1)$-invariant differential forms on $H_\mathbb{H}^n$ is isomorphic to the algebra of $Sp(n+1)$-invariant differential forms on $\mathbb{H}P^n$, thus to the cohomology of $\mathbb{H}P^n$, which is a truncated polynomial algebra on a single generator $\alpha$ of dimension $4$.  Our normalization for $\alpha$ makes it correspond to an integral generator for the cohomology of $\mathbb{H}P^n$ with suitable positivity properties.

We have an isomorphism $\mathbb{C}^{2n+2}\cong \mathbb{H}^{n+1}$ obtained by letting $i\in\mathbb{C}$ act on $\mathbb{H}^{n+1}$ by right multiplication by the quaternion $i$.  Explicitly, the complex coordinates $z_1,\dots , z_{n_1},w_1,\dots ,w_{n+1}$ on $\mathbb{C}^{2n+2}$ are related to the quaternionic coordinates $x_1,\dots x_{n+1}$ by $x_l = z_l + j w_l$ for $l = 1,\dots n+1$.   In terms of these complex coordinates the quaternionic Hermitian forms decompose as 
\begin{equation*}
 h_0^\mathbb{H}(X,Y) = h_0^\mathbb{C}(X,Y) + j a_0^\mathbb{C}(X,Y)
 \end{equation*}
 and
 \begin{equation*}
  h^\mathbb{H}(X,Y) = h^\mathbb{C}(X,Y) + j a^\mathbb{C}(X,Y),
 \end{equation*}
where $h_0^\mathbb{C}$ and $h^\mathbb{C}$ are complex Hermitian forms (conjugate linear in the first argument, complex linear in the second) and $a_0^\mathbb{C}, a^\mathbb{C}$ are complex bilinear alternating forms, in fact, complex symplectic forms, on $\mathbb{C}^{2n+2}$.   Explicitly, in terms of $x_l = z_l + jw_l$ and $y_l = u_l + jv_l$, we have
\begin{equation}
\label{eq-forms}
h_0^\mathbb{C}(X,Y) = \sum_1^{n+1}(\bar z_l  u_l + \bar w_l  v_l)\ \ \hbox{and} \ \  a_0^\mathbb{C}(X,Y) = \sum_1^{n+1} (z_lv_l - w_lu_l).
\end{equation}
The forms $h^\mathbb{C}(X,Y)$ and $a^\mathbb{C}(X,Y)$ are obtained by changing the sign of the $(n+1)$-st summand in the above formulas to the opposite sign, so 
\begin{equation}
\label{eq-herm-form}
h^\mathbb{C}(X,Y) = \sum_1^{n}(\bar z_l  u_l + \bar w_l  v_l) - (\bar z_{n+1}  u_{n+1} + \bar w_{n+1} v_{n+1})
\end{equation}
and
\begin{equation}
\label{eq-sympl-form}
a^\mathbb{C}(X,Y) = \sum_1^{n} (z_lv_l - w_lu_l) - (z_{n+1}v_{n+1} - w_{n+1}u_{n+1}).
\end{equation}

We obtain a natural map $\pi:\mathbb{C}P^{2n+1}\to \mathbb{H}P^n$ by assigning to a complex line $l$ the right quaternionic line $L$ it generates.  This map is a fibration with fiber the complex projective line $\mathbb{C}P^1$.  We let $D^n = \pi^{-1}(H_\mathbb{H}^n)$.   This is the open subset of $\mathbb{C}P^{2n+1}$ consisting of complex lines in $\mathbb{C}^{2n+2}\cong\mathbb{H}^{n+1}$ on which the form $h^\mathbb{C}$ is negative.  This inclusion gives $D^n$ the structure of a complex manifold, that fibers over the real manifold $H_\mathbb{H}^n$ with fiber $\mathbb{C}P^1$, and is called the {\em twistor space} of $H_\mathbb{H}^n$.   These fibers are complex submanifolds of $D^n$.    A complex line $l\in\mathbb{C}P^{2n+1}$ on the fiber over a quaternionic line $L\in H_\mathbb{H}^n$ (thus $l\subset L$) is the same as a right $\mathbb{H}$-linear and $h^\mathbb{C}$ isometric complex structure $J_L$ on $L$, namely the complex structure that is right multiplication by $i$ on $l$
  and right multiplication by $-i$ on the $h^\mathbb{C}$-orthogonal complement of $l$. 

As homogenous spaces we have $D^n = Sp(n,1)/Sp(n)\times U(1)$ while 
$$\mathbb{C}P^{2n+1} = Sp(n+1)/Sp(n)\times U(1).$$
 We have the commutative diagram:

$$
\begin{array}{rcl}
D^n & \subset & \mathbb{C}P^{2n+1} \\
 \pi\Big\downarrow & & \Big\downarrow\pi \\
H_\mathbb{H}^n & \subset & \mathbb{H}P^n
\end{array}
$$

The {\em vertical bundle} $\mathcal{V} = ker(d\pi)\subset TD^n$ tangent to the fibers of the projection is a $C^\infty$ sub-bundle of $TD^n$. It has a unique $Sp(n,1)$-invariant complement $\mathcal{H}$, called the {\em horizontal bundle}, which is a holomorphic sub-bundle of $TD^n$.  In terms of the canonical description of $T\mathbb{C}P^{2n+1}$ as $Hom(S, \mathbb{C}P^{2n_1}\times\mathbb{C}^{2n+2}/S)$, where $S$ is the tautological line sub-bundle of the trivial bundle $\mathbb{C}P^{2n+1}\times\mathbb{C}^{2n+2}$,  the sub-bundle $\mathcal{H}$ is the restriction to $D^n$ of  $Hom(S,S^\bot)$, where $S^\bot$ is the orthogonal complement of $S$ with respect to the complex symplectic form $a^\mathbb{C}$ of  (\ref{eq-sympl-form}).   Note that $Hom(S,S^\bot)$ is a holomorphic contact structure on $\mathbb{C}P^{2n+1}$, invariant under the group $Sp(a^\mathbb{C})$.

\begin{definition}
\label{def-integral-element}
Given $x\in D^n$, a linear subspace $W\subset \mathcal{H}_x$ is called an {\em integral element} if for all $X,Y\in W$, the vertical component $[X,Y]^\mathcal{V} = 0$.   
\end{definition}

Observe that this definition makes sense because, given any two vector fields $X,Y$ in $\mathcal{H}$, the value of $[X,Y]$ at $x$ depends only on the values of $X$ and $Y$ at $x$.  The geometric meaning of integral element is that, if $M\subset D^n$ is a horizontal submanifold, in other words, an integral submanifold of $\mathcal{H}$, and $x\in M$, then $T_xM\subset \mathcal{H}_x$ is an integral element.

By a {\em pseudo-Hermitian metric} on a complex manifold $M$ with complex structure $J$ we mean a non-degenerate inner product (not necessarily positive definite) on each tangent space that is invariant under $J$.   By a {\em pseudo-K\"ahler metric} we mean a pseudo-Hermitian metric whose associated $(1,1)$ form is closed.

\begin{lemma}
\label{lem-pseudo-metric}
\begin{enumerate}
\item The space $D^n$ has an indefinite pseudo-K\"ahler metric $g$ which is $Sp(n,1)$-invariant and which is negative definite on $\mathcal{V}$, is positive definite on  $\mathcal{H}$, and $\mathcal{V}$ and $\mathcal{H}$ are $g$-orthogonal.

\item  Let $\omega_{D^n}$ be the $(1,1)$ form associated to $g$.  Then  $\pi^*\alpha= \omega^2_{D^n} + d\beta$ for some $Sp(n,1)$-invariant $3$-form $\beta$.
\end{enumerate}
\end{lemma}

\begin{proof}

To prove the first part, note that $-h^\mathbb{C}$ is a positive Hermitian metric on the tautological sub-bundle $S$ over $D^n$.  The  form $\frac{i}{2\pi}\partial\bar\partial\log(-h^\mathbb{C})$ is the pull-back of a form $\omega_{D^n}$ on $D^n$, the Chern form of $S$.  It  is a closed, $Sp(n,1)$-invariant $(1,1)$-form on $D^n$ which is easily checked to be negative on $\mathcal{V}$ and positive on $\mathcal{H}$.   (Alternatively, since the canonical bundle of $D^n$ is a positive multiple of $S$ (as $Sp(n,1)$-homogeneous bundles on $D^n$),  one can quote (4.23) of \cite{GS} for the signature of this form.)

The form $\omega_{D^n}$ is then the $(1,1)$ form of an indefinite K\"ahler metric $g$ on $D^n$ that has all the asserted properties.    This proves the first assertion.

Now we prove the second assertion.  Whenever a group $G$ acts on a space $A$,  let us write $A^G$ for the subspace of $G$-invariant elements.  If $A$ denotes the algebra of differential forms, then the cohomology of invariant forms on $D^n$, $H^*(A(D^n)^{Sp(n,1)}, d)$ (where $d$ is the usual exterior derivative) is the same as the relative Lie-algebra cohomology $H^*(\mathfrak{sp}(n,1),\mathfrak{sp}(n)\oplus\mathfrak{u}(1),\mathbb{R})$.   Similarly, the cohomology $H^*(A(\mathbb{C}P^{2n+1})^{ Sp(n+1)},d)$ of $Sp(n+1)$-invariant forms on $\mathbb{C}P^{2n+1}$ is the same as the relative Lie algebra cohomology $H^*(\mathfrak{sp}(n+1), \mathfrak{sp}(n)\oplus\mathfrak{u}(1),\mathbb{R})$.  

With our choice of definitions for $Sp(n,1)$ and $Sp(n+1)$, we get that the complexification of $\mathfrak{sp}(n,1)$ is the algebra $\mathfrak{sp}(a^\mathbb{C})$ of infinitesimal isometries of the symplectic form $a^\mathbb{C}$   of (\ref{eq-sympl-form}) on $\mathbb{C}^{2n+2}$, while the complexitication of $\mathfrak{sp}(n+1)$ is the similarly defined algebra $\mathfrak{sp}(a_0^\mathbb{C})$ of (\ref{eq-forms}).  These two algebras do not coincide, but are conjugate in $\mathfrak{sl}(2n+2,\mathbb{C})$, say by conjugating by the linear isomorphism of $\mathbb{C}^{2n+2}$ that reverses the sign of the last coordinate and is the identity on all the others.   This isomorphism is the identity on the complexified subalgebras denoted $\mathfrak{sp}(n)\oplus\mathfrak{u}(1)$, hence induces an isomorphism between the relative Lie algebra cohomologies $H^*(\mathfrak{sp}(a^\mathbb{C}),(\mathfrak{sp}(n)\oplus\mathfrak{u}(1))^\mathbb{C},\mathbb{C})$ and $H^*(\mathfrak{sp}(a_0^\mathbb{C}),(\
 mathfrak{sp}(n)\oplus\mathfrak{u}(1))^\mathbb{C},\mathbb{C})$, hence an isomorphism of the complexifications  $H^*(A(D^n)^{Sp(n,1)})\otimes\mathbb{C}$ and $H^*(A(\mathbb{C}P^{2n+1})^{Sp(n+1)})\otimes\mathbb{C}$.    By compactness, the second computes the cohomology $H^*(\mathbb{C}P^{2n+1},\mathbb{C})$.  It follows that the real vector space  $H^4(A(D^n)^{Sp(n,1)})$ is one-dimensional, hence $\pi^*\alpha = c\  \omega_{D^n}^2 + d\beta$ for some real constant $c$ and some invariant $3$-form $\beta$.  Restricting to a totally geodesic $B^n$ and recalling the normalization of $\alpha$, we see that $c = 1$, and the lemma is proved. 

\end{proof}

\begin{remark} Observe that the space $A^2(D^n)^{Sp(n,1)}$ is two-dimensional, and its subspace of closed forms is one-dimensional, spanned by the form $\omega_{D^n}$.  This can easily be checked by fixing $x\in D^n$ and using the isomorphism
$$
A^*(D^n)^{Sp(n,1)}\to \Lambda^*(T_xD^n)^{Sp(n)\times U(1)}\cong \Lambda^*(\mathcal{V}_x\oplus\mathcal{H}_x)^{Sp(n)\times U(1)},
$$
where the first map is given by restriction.  It is not hard to see that 
$$
\Lambda^2(\mathcal{V}_x\oplus\mathcal{H}_x)^{Sp(n)\times U(1)}\cong\Lambda^2(\mathcal{V}_x)^{U(1)}\oplus\Lambda^2(\mathcal{H}_x)^{Sp(n)\times U(1)},
$$
and that each summand is one-dimensional.  Since there are no invariant one-forms and the second cohomology is one-dimensional, it follows that the space of closed forms is one dimensional, and necessarily spanned by $\omega_{D^n}$ of Lemma~\ref{lem-pseudo-metric}.  Equivalently, the space of invariant Hermitian metrics on $D^n$ is two-dimensional, and the space of invariant pseudo-K\"ahler metrics is one-dimensional, consisting of non-zero multiples of the metric $g$ of Lemma~\ref{lem-pseudo-metric}.
\end{remark}

Next,  we need to describe the totally geodesic embeddings of complex hyperbolic space $B^m = H_\mathbb{C}^m$ in quaternionic hyperbolic space $H_\mathbb{H}^n$.  See \cite{Be} (for $\mathbb{H}P^n$, with similar results for $H_\mathbb{H}^n$), the thorough treatment of geodesic subspaces in \cite{CG}, or the  discussion in \cite{CT} for details.   The result is that any totally geodesic $B^m$ in $H_\mathbb{H}^n$ for $m\ge 2$ is obtained as follows:  First $m\le n$ and  there exists an $m+1$-dimensional right quaternionic subspace $V\subset \mathbb{H}^{n+1}$ so that the restriction of $Re(h^\mathbb{H})$ to $V$ has signature $(4m,4)$, and there exists a right quaternionic linear complex structure $J$ on $V$ which preserves $h^\mathbb{H}$ and so that the embedding of $B^m$ in $H_\mathbb{H}^n$ consists of those right quaternionic lines in $V$ that are invariant under $J$.   (If $m=2$ this construction produces totally geodesic $B^1$'s, but not all, because $H_\mathbb{H}^1$ is the s
 ame as real hyperbolic $4$-space and $B^1$ is the same as the real hyperbolic plane.  Real hyperbolic $4$-space has many more totally geodesic real hyperbolic planes than the ones arising in this way.)

These totally geodesic $B^m$'s in $H_\mathbb{H}^n$ have very canonical horizontal lifts to $D^n$ that we now describe.   First, consider the case $m = n$.   Given a totally geodesic $B^n\subset H_\mathbb{H}^n$, there is a right $\mathbb{H}$-linear complex structure $J$ on $\mathbb{H}^{n+1}$ preserving $h^\mathbb{H}$ so that $B^n$ consists of all the $J$-invariant lines.  Consider the isomorphism $\mathbb{H}^{n+1}\cong\mathbb{C}^{2n+2}$ as before, with $i$ acting by right quaternion multiplication.  Taking account of the Hermitian forms, it is  more accurate to use the notation $\mathbb{H}^{(n,1)}\cong\mathbb{C}^{(2n,2)}$, which we use from now on.   Let $V$ denote the $i$-eigenspace of $J$ acting on $\mathbb{C}^{(2n,2)}$.   It easy to check that $Vj$ is the $-i$-eigenspace of $J$, that we have a direct sum decomposition $\mathbb{C}^{(2n,2)} = V\oplus Vj$ into $h^\mathbb{C}$-orthogonal and $a^\mathbb{C}$-isotropic subspaces so that the restriction of $h^\mathbb{C}$ to each sub
 space has signature $(n,1)$.  Passing to complex projectivization of negative lines, we obtain two holomorphic embeddings of $B^n$ in $D^n$, each projecting (under right quaternionic projectivization) to the original geodesic embedding of $B^n$ in $H_\mathbb{H}^n$ (each projection being holomorphic with respect to each of two conjugate complex structures).  Moreover, since each of the eigenspaces $V, Vj$ is $a^\mathbb{C}$-isotropic, it follows easily that each of these embeddings of $B^n$ in $D^n$ is horizontal.  In fact, each of these two $B^n$'s is a Legendrian submanifold (integral element of maximum dimension)  of the contact structure (horizontal sub-bundle $\mathcal{H}$) of $D^n$.  

Note that each of these two embeddings is fixed by the element of $Sp(n,1)$ that is right multiplication by $i$ on $V$ and by $-i$ on $Vj$.  If follows (and we explain in more detail below)  that each is a totally geodesic submanifold of $D^n$ in the $Sp(n,1)$-invariant pseudo-K\"ahler metric $g$ defined in Lemma~\ref{lem-pseudo-metric}.  Note also that  the restriction of $\omega_{D^n}$ to each $B^n$ coincides with the K\"ahler form $\omega_{B^n}$.

Finally, if $m<n$, we use the same construction with a suitable $\mathbb{H}^{(m,1)}\subset\mathbb{H}^{(n,1)}$.  As remarked above, in case $m=1$ we do not obtain all geodesic embeddings $B^1\subset H_\mathbb{H}^n$ in this way, but we do obtain a special class of embeddings, namely the ones covered by holomorphic horizontal geodesic embeddings of $B^1$ in $D^n$.    

We now explain what we mean by {\em horizontal totally geodesic complex submanifolds } of $D^n$.   We consider a   complex manifold $(M,J)$ with a pseudo-Hermitian metric $g$  (by which
we mean a pseudo-Riemannian metric $g$ which is $J$-invariant: $g(JX,JY) = g(X,Y)$ for all tangent vectors $X,Y$) and  a metric complex connection $\nabla$ (not necessarily torsion-free), meaning that $\nabla J =0$ and $\nabla g=0$..  Let $N\subset M$ be a complex submanifold which is never isotropic, in the sense that for all $x\in N$, $T_xN$ contains no null vectors for $g$.  Then $N$ is said to be {\em totally geodesic} if $TN$ is invariant under $\nabla$.   This means that for all vector fields $X,Y$ tangent to $N$, $\nabla_XY$ is also tangent to $N$.  An equivalent formulation is in terms of the {\em second fundamental form} $\alpha(X,Y)$, defined by $\alpha(X,Y) = (\nabla_X Y)^\bot$, where $\bot$ denotes the $g$-orthogonal complement.  If $TN$ contains no null vectors, then $N$ is totally geodesic if and only if $\alpha = 0$.   Note that if $\nabla$ has torsion, then $\alpha(X,Y)$ need not be symmetric in $X$ and $Y$.

We list the main properties that we need of the above embeddings.

\begin{lemma}
\label{lem-embeddings}
\begin{enumerate}

\item Each of the embeddings $B^m\subset D^n$ just described is a component of the fixed point set in $D^n$ of an element of $Sp(n,1)$.   

\item In particular, each such embedding is totally geodesic with respect to any pair $(g,\nabla)$ of  $Sp(n,1)$-invariant Hermitian metric $g$ and metric connection $\nabla$.

\item For any pair $(g,\nabla)$ as above, every horizontal holomorphic embedding of $B^m$ in $D^n$ that is totally geodesic  in the sense of vanishing second fundamental form,  is one of the embeddings described above.

\end{enumerate}
\end{lemma}

\begin{proof}

Fix $m\le n$.  Each of the embeddings in question is equivalent under $Sp(n,1)$ to the embedding of $B^m$ in $D^n$ resulting from the embedding 
$$
(z_1,\dots,z_m,z_{m+1})\to (z_1,\dots,z_m,0,\dots,0,z_{m+1})
$$
of $\mathbb{C}^{(m,1)}$ in $\mathbb{H}^{(n,1)}$.   We still use the quaternionic and related complex  coordinates $x_l = z_l + j w_l$ in $\mathbb{H}^{(n,1)}\cong \mathbb{C}^{(2n,2)}$.  The image of $B^m$ is a component of the fixed point set of the element

$$
(x_1,\dots,x_{n+1})\to (ix_1,\dots,ix_m,-x_{m+1},\dots,-x_n,ix_{n+1})
$$
of the isotropy subgroup $Sp(n)\times U(1)\subset Sp(n,1)$.  (There is one other component in $D^n$, the $-i$ eigenspace.  The $-1$ eigenspace is disjoint from $D^n$.)  This proves the first part of the Lemma.  

If $g$ is any invariant Hermitian metric on $D^n$, then, by the discussion in Lemma~\ref{lem-pseudo-metric}, the restriction of $g$ to $\mathcal{H}$ is definite, thus horizontal submanifolds contain no null vectors for $g$.  The above definition of totally geodesic submanifolds and characterization in terms of second fundamental form holds.    If $M$ is a component of the fixed point set of a map that preserves $g$ and $\nabla$, then no normal vector to the fixed point set is fixed, thus the second fundamental form $\alpha(X,Y)$ must be fixed, thus $\alpha = 0$, and $M$ is totally geodesic.   This reasoning applies to our embeddings and proves the second assertion of the Lemma.

Now take a pair $(g,\nabla)$ as in the second part, and a holomorphic horizontal embedding  of $B^m$ in $D^n$ with vanishing second fundamental form.  Such an embedding is uniquely determined by its tangent space at one point.  But it is easy to check that there is one of our special embeddings  tangent to each complex (meaning invariant under the complex structure $J$ of $D^n$) integral element of $\mathcal{H}_x$.  Since the tangent space of every horizontal complex submanifold of $D^n$ is a complex integral element, our construction has found all  holomorphic horizontal  embeddings that are totally geodesic in the usual differential-geometric sense.

\end{proof}

\section{Schwarz Lemma for Horizontal Holomorphic Maps}

A holomorphic map $F:B^m\to D^n$ is called {\em horizontal} if for all $x\in B^m$, $dF(T_xB^m)\subset \mathcal{H}_{F(x)}$.  It is well known that horizontal holomorphic maps satisfy a Schwarz Lemma, see \cite{ CMP, GS}  In addition to the well-known inequality, we need a discussion of the case of equality.  Since we could not find any discussion in the literature of the precise result we need,  we prove the Schwarz Lemma in some detail.   We follow the method of \cite{CMP}, which in turn follows \cite{Ko}.   Both \cite{CMP,GS} use a definite Hermitian metric on $D^n$ which is necessarily not K\"ahler, and its Chern connection.  In order to relate equality to totally geodesic maps we find it more convenient to use the pseudo-K\"ahler metric of Lemma~\ref{lem-pseudo-metric} and its Levi-Civita connection.  The two give of course equivalent results for horizontal holomorphic maps as we explain later.

Let $(g,\nabla)$ be a pair of  invariant pseudo-Hermitian metric and metric connection on $D^n$.   Assume that $g|\mathcal{H}$ is positive definite, and normalize $g$ so that its restriction to every horizontal geodesic $B^m\subset D^n$ as in the last section is the metric of constant holomorphic sectional curvature $-1$.  Moreover, assume that $\nabla^{(0,1)}$ is the $\bar \partial$-operator of the complex structure on $TD^n$.  Two examples of such a pair are:
\begin{enumerate}
\item The pseudo-K\"ahler metric of Lemma~\ref{lem-pseudo-metric}, let's denote this metric by $g^K$ and let's denote its Levi-Civita connection by $\nabla^K$.   Then $\nabla^K$ is also the Chern connection on $TD^n$ preserving $g^K$, namely $(\nabla^K)^{(0,1)} =\bar\partial$.

\item The positive definite metric used in \cite{CMP,GS} using the Killing form and the Cartan involution.  Let us denote this metric by $g^C$ and its corresponding Chern connection by $\nabla^C$.

\end{enumerate}

Then the restriction of $\nabla$ to $B^m$ is necessarily the canonical connection $\nabla$ on $B^m$, which is both the Levi-Civita connection and the Chern connection, as is the case for any K\"ahler metric.

Given such a compatible pair $(g,\nabla)$, we can define the curvature tensor $R = R^\nabla$ in the usual way, and we can define the {\em horizontal holomorphic sectional curvatures} to be the numbers $R(X,JX,X,JX)$ defined by  
$$
R(X,JX,X,JX) = g(R(X,JX)JX,X)\ \hbox{for each}\ X\in \mathcal{H}, \  g(X,X) = 1.
$$

\begin{lemma}  For any compatible pair $(g,\nabla)$ as just defined, the horizontal holomorphic sectional curvatures  are constant, identically equal to $-1$.

\end{lemma}

\begin{proof}

Let $x\in D^n$.  Then any complex one-dimensional subspace of  $\mathcal{H}_x$ is the real span of $X,JX$ for some unit vector $X\in\mathcal{H}_x$.  Since this is the tangent space to a holomorphic, horizontal totally geodesic embedding of $B^1$ in $D^n$, we have, by the usual argument, that the holomorphic sectional curvature in $D^n$ is the same as in $B^1$, namely $-1$.

\end{proof}

Observe that the isotropy group at $x\in D^n$ acts transitively on the complex projective space of the space $\mathcal{H}_x$, since this action is isomorphic to the action of $Sp(n)\times U(1)$ on $\mathbb{C}P^{2n-1}$, which is indeed transitive.  This already implies that the horizontal holomorphic sectional curvatures must be constant.

\begin{lemma}
Let $(g,\nabla)$ be the compatible pair $(g^K,\nabla^K)$ of pseudo-K\"ahler metric and compatible torsion-free connection.   Let $M\subset D^n$ be a horizontal submanifold.   Let $\alpha:T_xM\otimes T_xM\to T_xM^\bot$ be its second fundamental form.   Then $\alpha(X,Y)\in \mathcal{H}_x$ for all $X,Y,\in T_xM$.

\end{lemma}
\label{lem-horizontal}

\begin{proof}
Observe that if $X,Y$ are horizontal vector fields, then the value of the vertical component $(\nabla_X Y)^\mathcal{V}$ at $x$ depends just on the values of $X, Y$ at $x$, thus we get a well defined tensor  $\mathcal{H}_x\otimes\mathcal{H}_x\to\mathcal{V}_x$.  If $X,Y$ span an integral element of $\mathcal{H}$, namely, if $X^\mathcal{V},Y^\mathcal{V}$ and $[X,Y]^\mathcal{V}$ all vanish, see Definiton~\ref{def-integral-element}, we see that $(\nabla_X Y)^\mathcal{V}$ is symmetric in $X$ and $Y$,  since the symmetry of the connection $\nabla = \nabla^K$ gives 
$$
(\nabla_X Y)^\mathcal{V} - (\nabla_Y X)^\mathcal{V} = [X,Y]^\mathcal{V} = 0.
$$
Since $(\nabla_X X)^\mathcal{V} = 0$, we see that $(\nabla_X Y)^\mathcal{V} = 0$ on all integral elements $X,Y$.

		Suppose that $M\subset D^n$ is a horizontal submanifold and that $\alpha$ is its second fundamental form.   Then $\alpha(X,Y) = (\nabla_X Y)^\bot = \alpha(X,Y)^\mathcal{H} + \alpha(X,Y)^\mathcal{V}$, where $\alpha(X,Y)^\mathcal{V} = (\nabla_X Y)^\mathcal{V} = 0$ since $X,Y$ form an integral element of $\mathcal{H}$.   Thus $\alpha(X,Y)\in\mathcal{H}$.

\end{proof}

\begin{theorem}
\label{thm-schwarz}

Let $F:B^m\to D^n$ be a horizontal holomorphic map , and let $\omega_{B^m}, \omega_{D^n}$ be as above.  Then $F^*\omega_{D^n} \le \omega_{B^m}$.  Equality holds at every point if and only if $F$ is a horizontal holomorphic geodesic embedding of  $B^m$ in $D^n$.

\end{theorem}

\begin{proof}

We follow the method explained in \cite{CMP, Ko}.   First we treat the case $m =1$, to which the general case is easly reduced.   

Suppose $F:B^1\to D^n$ is a horizontal holomorphic map.  Then for all $x\in B^1$, $dF(T_xB^1)\subset \mathcal{H}_{F(x)}$, on which the metric $g = g^K$ is positive definite and on which the holomorphic sectional curvature is $-1$.   If we write $F^*\omega_{D^n} = u\  \omega_{B^1}$ for a non-negative smooth function $u$ on $B^1$, then, by the method of \S 2 of Chapter I and \S2 of Chapter III  of \cite{Ko}, see also \S13.4 of \cite{CMP},  we may assume that $u$ attains its maximum, and a computation of its Laplacian at the maximum gives the inequality $u\le 1$, hence the desired inequality $F^*\omega_{D^n}\le \omega_{B^1}$

If equality holds at every point, then $F$ is an isometric immersion.  Locally $F$ is an isometric embedding.  Let $M\subset D^n$ be the image of such a local embedding, and let $\alpha$ be its second fundamental form.    Since the metric $g$ is pseudo-K\"ahler, the usual K\"ahler identities hold, and we get the formula, for any unit vector $X\in T_xM$, 
$$
R^M(X,JX,X,JX) = R^{D^n}(X,JX,X,JX) - 2g(\alpha(X,X),\alpha(X,X)),
$$
see Proposition 9.2 of Chapter IX of \cite{KN}.  
Since both holomorphic sectional curvatures are equal (to $-1$), we must have $g(\alpha(X,X),\alpha(X,X)) =0$.   Since, by Lemma~\ref{lem-horizontal},  $\alpha(X,X)\in \mathcal{H}$, on which $g$ is positive definite, it follows that $\alpha(X,X) = 0$, hence $M$ is a horizontal complex  totally geodesic submanifold, thus, by Lemma~\ref{lem-embeddings}, it must be one of our special embeddings.  This proves the case $m = 1$.

To prove the case $m>1$, take any $x\in B^m$ and any complex direction at $x$ and take the geodesic $B^1\subset B^m$ through $x$ in this direction..  Restricting $F$ to this $B^1$ we get the desired inequality.  If equality holds for all $x$ and all directions, then every geodesic $B^1\subset B^m$ must be geodesically isometrically embedded in $D^n$, from which we get the vanishing of the second fundamental form of the embedding of $B^m$, hence the latter is  a horizontal holomorphic totally geodesic embedding, hence as described in Lemma~\ref{lem-embeddings}.

\end{proof}

\begin{remark}
A parallel proof, with a somewhat different discussion in the case of equality, would hold using the definite Hermitian (non-K\"ahler) metric mentioned above.
\end{remark}

\section{Proof of Theorem~\ref{thm-main}}

Let $X = \Gamma\backslash B^m$ be a compact manifold and let $\rho:\Gamma\to Sp(n,1)$ be a representation.  If $c(\rho) = 0$ there is nothing to prove.   If $c(\rho) \ne 0$, then $\rho$ is a reductive representation, in fact, its image is Zariski-dense in $Sp(n,1)$.   The reason is that otherwise its image lies in a parabolic subgroup, and in this case this makes $c(\rho) = 0$, see the similar argument in the last paragraph of the proof of Theorem 6.1 of  \cite{Co}. Then,  by Corlette's existence theorem for harmonic metrics (equivalently, equivariant harmonic maps), see \S 3 of  \cite{Co},  (see also \cite{Do}), there is an equivariant harmonic map $f:B^m\to H_\mathbb{H}^n$.  Moreover, since $f^*\alpha \ne 0$, the rank of $f$ is at least four.  By Theorem 6.1 of \cite{CT} there is an equivariant horizontal holomorphic map $F:B^m\to D^n$ that lifts $f$, that is, $f = \pi\circ F$. Since the form $\beta$ of Lemma~\ref{lem-pseudo-metric} descends to a form on $X$, we get  that, 
 on $X$,  $f^*\alpha$ is cohomologous to $F^*\omega_{D^n}^2$.  (We do not distinguish between  $\Gamma$-invariant forms on $B^m$ and forms on $X$, in particular we write either $\omega_X$ or $\omega_{B^m}$ as convenient).

Thus we get 
$$
c(\rho) = \int_X \ F^*\omega_{D^n}^2\wedge \omega_X^{m-2}.
$$
Since $F^*\omega_{D^n}\le \omega_X$, we get the inequality.   If equality holds, then we must have a pointwise  equality $F^*\omega_{D^n}\wedge\omega_X^{n-2}  = \omega_X^n$.  This in turn implies  a pointwise equality $F^*\omega_{D^n} = \omega_X$.  This standard fact can be seen, for example, from the easily verified formula
$$
 F^*\omega_{D^n}^2\wedge \omega_X^{n-2} = \frac{2}{n(n-1)} Tr(\Lambda^2\phi)\  \omega_X^n,  
 $$
where, for each $x\in B^m$,  $\phi = (d_xF)^*(d_xF):T_xB^m\to T_xB^m$ and $(d_xF)^*$ denotes the Hermitian adjoint of $d_xF:T_xB^m\to \mathcal{H}_{F(x)}$.   Since the eigenvalues of $\phi$ are non-negative and (by the inequality part of the Schwarz lemma) at most one, equality holds if and only if all the eigenvalues are one, which is equivalent to $F^*\omega_{D^n} = \omega_{B^m}$.   Thus, by the equality part of Theorem~\ref{thm-schwarz}, 
 $F$ is a totally geodesic horizontal holomorphic embedding, hence $f$ is a totally geodesic isometric embedding, hence $\rho$ is a totally geodesic representation.

\end{document}